\documentclass[reqno]{amsart}

\usepackage{tikz}
\usepackage{caption}
\usepackage{subcaption}

\usepackage{amsmath}
\usepackage{amssymb}
\usepackage{amsfonts}

\usepackage{amsthm}

\newtheorem*{thm}{Theorem}
\newtheorem{lemma}{Lemma}

\newtheorem{proposition}{Proposition}

\usepackage{hyperref}
\hypersetup{colorlinks=true, linkcolor=blue}

\begin{document}

\title[]{Graph decomposition via edge edits\\ into a union of regular graphs}
\author[]{Tony Zeng}
\address{Department of Mathematics, University of Washington, Seattle, WA 98195, USA} \email{txz@uw.edu}


\begin{abstract}
	Suppose a finite, unweighted, combinatorial graph \(G = (V,E)\) is the union of several (degree-)regular graphs which are then additionally connected with a few additional edges. \(G\) will then have only a small number of vertices $v \in V$ with the property that one of their neighbors $(v,w) \in E$ has a higher degree $\deg(w) > \deg(v)$. We prove the converse statement: if a graph has few vertices having a neighbor with higher degree and satisfies a mild regularity condition, then, via adding and removing a few edges, the graph can be turned into a disjoint union of (distance-)regular graphs. The number of edge operations depends on the maximum degree and number of vertices with a higher degree neighbor but is independent of the size of $|V|$.
\end{abstract}

\maketitle

\section{Introduction and Results}

\subsection{Graph decompositions}

If \(G = (V, E)\) is a graph obtained by taking some disjoint regular graphs (possibly of different degrees) with many vertices, adding a few edges to connect these components, and then possibly removing a few edges, \(G\) is still `approximately' regular in the sense that most vertices \(v \in V\) have the same degree as all their neighbors \(u \in N(v)\).

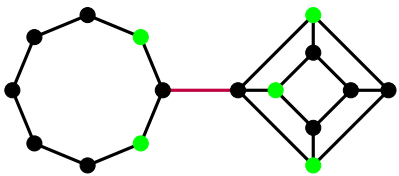
\begin{figure}[!ht]
	\centering
	\begin{tikzpicture}[scale=1]
		\draw[purple, very thick] (1, 0) -- (2, 0);
            \foreach \a in {1,...,8}{
                \draw[very thick] (360 / 8 * \a:1) -- (360 / 8 * \a + 45:1);
            }
		\foreach \a in {2,3,4,5,6,8}{
			\filldraw (360 / 8 * \a:1) circle (0.1) node (\a) {};
		}
		\foreach \a in {1,...,4}{
			\draw[very thick] (3, 0) + (360 / 4 * \a:1) -- +(360 / 4 * \a + 90:1);
			\draw[very thick] (3, 0) + (360 / 4 * \a:0.5) -- +(360 / 4 * \a + 90:0.5);
			\draw[very thick] (3, 0) + (360 / 4 * \a:1) -- +(360 / 4 * \a:0.5);
		}
            \filldraw (3, 0.5) circle (0.1);
            \filldraw (3, -0.5) circle (0.1);
            \filldraw (3.5, 0) circle (0.1);
            \filldraw (4, 0) circle (0.1);
            \filldraw (2, 0) circle (0.1);
            \filldraw[green] (360 / 8 * 1:1) circle (0.1);
            \filldraw[green] (360 / 8 * 7:1) circle (0.1);
            \filldraw[green] (3, 0) + (-0.5, 0) circle (0.1);
            \filldraw[green] (3, 0) + (0, 1) circle (0.1);
            \filldraw[green] (3, 0) + (0, -1) circle (0.1);
	\end{tikzpicture}
	\caption{A 2-regular and 3-regular graph connected by an edge (in purple). Deleting this edge yields a union of regular graphs. The graph has 5 separation vertices (shown in green). Later diagrams will color separation vertices in green, deleted edges in red, and added edges in blue. We will study whether having a small number of separation vertices is indicative of $G$ being close to a union of regular graphs.}
\end{figure}

We define a vertex $v \in V$ to be a \textbf{separation vertex} if it has a neighbor $(v,w) \in E$ with a higher degree $\deg(w) > \deg(v)$. As an additional piece of terminology, we will speak of \textbf{modifying} an edge which will mean that we add or remove an edge. Throughout the paper, all our graphs are simple, we do not consider multiple edges between the same pair of vertices. If we start with a regular graph and modify some edges, that is add or remove some edges, the number of separation vertices that arises in the modified graph can be  bounded from above by the number of modifications. 
\begin{proposition}
	Let \(G = (V, E)\) be a \(d\)-regular graph. Let \(G'\) be \(G\) after adding or removing \(k\) edges. \(G'\) has at most \(2dk\) separation vertices.
	\label{prop:regmod}
\end{proposition}

We emphasize that the bound is completely independent of the actual size of the graph \(|V|\). This is not too surprising: adding or removing edges has a local effect on the degrees of vertices and their neighbors. We are interested in understanding whether some kind of converse statement is true.

\begin{quote}
	\textbf{Question.} Suppose \(G = (V, E)\) is a graph with a small number of separation vertices. Can \(G\), by adding or removing a few edges, be turned into a disjoint union of regular graphs?
\end{quote}

We show that results of this type indeed exist: a small number of separation vertices is indicative of $G$ being close (up to adding or deleting a few edges) to a disjoint union of (degree-)regular graphs (possibly of different degrees).

\begin{figure}[ht!]
	\centering
	\begin{subfigure}{\textwidth}
		\centering
		\begin{tikzpicture}[scale=0.75]
			\foreach \a in {-6,-4,-2,-1,1,2,4,6}{
				\draw[very thick] (\a, 1) -- (\a, -1);
			}
			\draw[very thick] (-5, -1) -- (-6, 1) -- (-3.5, 1);
			\draw[very thick] (-5, 1) -- (-6, -1) -- (-3.5, -1);
			\draw[very thick] (5, -1) -- (6, 1) -- (3.5, 1);
			\draw[very thick] (5, 1) -- (6, -1) -- (3.5, -1);
			\draw[very thick] (-2.5, 1) -- (2.5, 1);
			\draw[very thick] (-2.5, -1) -- (2.5, -1);
			\draw[very thick] (0, -2) -- (0, -1);
			\draw[very thick] (0, 2) -- (0, 1);
			\draw[very thick] (-6, -3) -- (-6, -2) -- (-3.5, -2);
			\draw[very thick] (6, -3) -- (6, -2) -- (3.5, -2);
			\draw[very thick] (-2.5, -2) -- (2.5, -2);
			\draw[very thick] (-6, 3) -- (-6, 2) -- (-3.5, 2);
			\draw[very thick] (6, 3) -- (6, 2) -- (3.5, 2);
			\draw[very thick] (-2.5, 2) -- (2.5, 2);
			\draw[very thick] (1, 2) -- (1, 3) -- (2.5, 3);
			\draw[very thick] (3.5, 3) -- (6, 3);
			\foreach \a in {-6,-5,-4,-2,-1,0,1,2,4,5,6}{
				\filldraw (\a, 2) circle (0.1);
				\filldraw (\a, 1) circle (0.1);
				\filldraw (\a, -1) circle (0.1);
				\filldraw (\a, -2) circle (0.1);
			}
			\foreach \a in {1,2,4,5,6}{
				\filldraw (\a, 3) circle (0.1);
			}
			\filldraw[green] (-6, -3) circle (0.1);
			\filldraw[green] (-6, 3) circle (0.1);
			\filldraw[green] (6, -3) circle (0.1);
			\filldraw[green] (1, 3) circle (0.1);
			\filldraw[green] (2, 2) circle (0.1);
			\node at (-3, 1) {\(\cdots\)};
			\node at (3, 1) {\(\cdots\)};
			\node at (-3, -1) {\(\cdots\)};
			\node at (3, -1) {\(\cdots\)};
			\node at (-3, -2) {\(\cdots\)};
			\node at (-3, 2) {\(\cdots\)};
			\node at (3, -2) {\(\cdots\)};
			\node at (3, 2) {\(\cdots\)};
			\node at (3, 3) {\(\cdots\)};
			\node at (-3, 0) {\(\cdots\)};
			\node at (3, 0) {\(\cdots\)};
		\end{tikzpicture}
	\end{subfigure}
	\caption{A family of graphs \(G\) with maximal degree \(3\) and 5 separation vertices (green). Can it be decomposed into a union of regular graphs by adding or deleting few edges? Can this be done in a number of operation independently of $|V|$?}
\end{figure}

 To make things precise, we will introduce a metric $d(G,H)$ on graphs. In our setting, each graph $G$  is uniquely determined by its adjacency matrix \(A \in \mathbb{R}^{n \times n}\), where $A_{ij} =1$ if $(i,j) \in E$ and $A_{ij} = 0$ otherwise. We use adjacency matrices to induce a natural metric on graphs: if $G$ and $H$ are two graphs on \(n\) vertices, we could wonder in how many places their adjacency matrices \(A_1, A_2\) differ. We never care about the actual labeling of the vertices, so the natural definition is thus up to a permutation of the entries of the matrix and
 \begin{align*}
     d(G, H) &= \frac{1}{2}\min_{\pi \in S_n} \sum_{i, j = 1}^n |(A(G))_{i,j} - (A(H))_{\pi(i),\pi(j)}| \\
     &= \frac{1}{2}\min_{\pi \in S_n} \Vert A(G) - P_\pi A(H) P_\pi^T \Vert_F^2,
 \end{align*}
where \(P_\pi\) is the permutation matrix associated to \(\pi\) and $\| \cdot \|_F$ denotes the Frobenius norm. It is easy to see that this induces a metric on the graphs on \(n\) vertices. We can now illustrate the spirit of our main result in a very simple special case.

\begin{proposition}
	\label{prop:spirit}
	Let \(G\) be a graph with maximal degree \(\Delta = \max_{v \in V} \deg(v) = 3\). If \(G\) has \(k\) separation vertices, then there exists a graph \(H\) that is the union of the (degree-)regular graphs which is close to $G$ 
	\[ d(G, H) \le 4k + 9. \]
\end{proposition}
It is immediately clear that this result is optimal up to constants, as there must be some modification affecting each separation vertex. 
It is also clear that, in general, these constants are going to depend on the maximal degree is necessary since the maximal degree controls how many neighboring vertices can be affected by modifying an edge. 
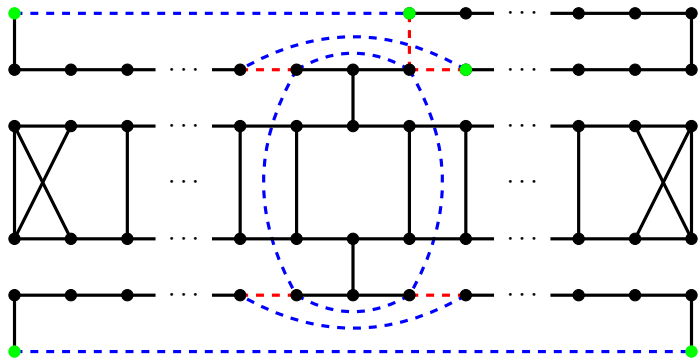
\begin{figure}[ht!]
	\centering
	
	\begin{subfigure}{\textwidth}
		\centering
		\begin{tikzpicture}[scale=0.75]
			\draw[blue, dashed, very thick] (-6, -3) -- (6, -3);
			\draw[blue, dashed, very thick] (-6, 3) -- (1, 3);
			\draw[blue, dashed, very thick] (1, -2) to[bend left=30] (-1, -2);
			\draw[blue, dashed, very thick] (1, 2) to[bend left=-30] (-1, 2);
			\draw[blue, dashed, very thick] (1, 2) to[bend left=30] (1, -2);
			\draw[blue, dashed, very thick] (-1, 2) to[bend left=-30] (-1, -2);
			\draw[blue, dashed, very thick] (2, 2) to[bend left=-30] (-2, 2);
			\draw[blue, dashed, very thick] (2, -2) to[bend left=30] (-2, -2);
			\foreach \a in {-6,-4,-2,-1,1,2,4,6}{
				\draw[very thick] (\a, 1) -- (\a, -1);
			}
			\draw[very thick] (-5, -1) -- (-6, 1) -- (-3.5, 1);
			\draw[very thick] (-5, 1) -- (-6, -1) -- (-3.5, -1);
			\draw[very thick] (5, -1) -- (6, 1) -- (3.5, 1);
			\draw[very thick] (5, 1) -- (6, -1) -- (3.5, -1);
			\draw[very thick] (-2.5, 1) -- (2.5, 1);
			\draw[very thick] (-2.5, -1) -- (2.5, -1);
			\draw[very thick] (0, -2) -- (0, -1);
			\draw[very thick] (0, 2) -- (0, 1);
			\draw[very thick] (-6, -3) -- (-6, -2) -- (-3.5, -2);
			\draw[very thick] (6, -3) -- (6, -2) -- (3.5, -2);
			\draw[very thick] (-1, -2) -- (1, -2);
			\draw[very thick] (-2.5, -2) -- (-2, -2);
			\draw[very thick] (2.5, -2) -- (2, -2);
			\draw[red, dashed, very thick] (-2, 2) -- (-1, 2);
			\draw[red, dashed, very thick] (2, 2) -- (1, 2);
			\draw[very thick] (-6, 3) -- (-6, 2) -- (-3.5, 2);
			\draw[very thick] (6, 3) -- (6, 2) -- (3.5, 2);
			\draw[very thick] (-1, 2) -- (1, 2);
			\draw[very thick] (-2.5, 2) -- (-2, 2);
			\draw[very thick] (2.5, 2) -- (2, 2);
			\draw[red, dashed, very thick] (-2, -2) -- (-1, -2);
			\draw[red, dashed, very thick] (2, -2) -- (1, -2);
			\draw[red, dashed, very thick] (1, 2) -- (1, 3);
			\draw[very thick] (1, 3) -- (2.5, 3);
			\draw[very thick] (3.5, 3) -- (6, 3);
			\foreach \a in {-6,-5,-4,-2,-1,0,1,2,4,5,6}{
				\filldraw (\a, 2) circle (0.1);
				\filldraw (\a, 1) circle (0.1);
				\filldraw (\a, -1) circle (0.1);
				\filldraw (\a, -2) circle (0.1);
			}
			\foreach \a in {1,2,4,5,6}{
				\filldraw (\a, 3) circle (0.1);
			}
			\filldraw[green] (-6, -3) circle (0.1);
			\filldraw[green] (-6, 3) circle (0.1);
			\filldraw[green] (6, -3) circle (0.1);
			\filldraw[green] (1, 3) circle (0.1);
			\filldraw[green] (2, 2) circle (0.1);
			\node at (-3, 1) {\(\cdots\)};
			\node at (3, 1) {\(\cdots\)};
			\node at (-3, -1) {\(\cdots\)};
			\node at (3, -1) {\(\cdots\)};
			\node at (-3, -2) {\(\cdots\)};
			\node at (-3, 2) {\(\cdots\)};
			\node at (3, -2) {\(\cdots\)};
			\node at (3, 2) {\(\cdots\)};
			\node at (3, 3) {\(\cdots\)};
			\node at (-3, 0) {\(\cdots\)};
			\node at (3, 0) {\(\cdots\)};
		\end{tikzpicture}
	\end{subfigure}
	\caption{Proposition 2 applied to the example from Figure 2. Adding the blue edges and deleting the red edges shows that there exists a regular graph $H$ with $d(G,H) = 13$. We note that the distance is independent of the number of vertices $\# V$.}
\end{figure}

\subsection{Main Result} We can now state our main result.
Our main results states, informally put, that if a graph $G$ has a small number of separation vertices and a sufficient number of vertices of small degrees then there is a union $H$ of regular graphs that is close to $G$. Everything is quantitative and independent of $|V|$.

\begin{thm}
	\label{thm:main}
	Suppose we have a graph \(G = (V, E)\) whose vertices have maximal degree \(\Delta\). If \(G\) has \(k\) separation vertices and at least \(2d + d^2 + d^3 + k \cdot (\Delta!)\) vertices of degree \(d\) for \(3 < d \le \Delta\), then there exists a union of regular graphs $H$ and 
	\[ d(G, H) \le 4(\Delta + 1)! \cdot k  + 5\Delta^2. \]
\end{thm}
\noindent \textbf{Remarks.}
\begin{enumerate}
\item The statement does not in any way depend on the total number of vertices \(|V|\), all conditions and conclusions only depend on $\Delta$ and $k$.
\item Simple examples show that the best one can hope for is a linear estimate in $k$ and in this sense our result is optimal. However, how the constants depend on $\Delta$ is surely suboptimal. It is easy to see that in the statement $d(G,H) \leq A k + B$, one needs the constant $A$ to grow at least linearly in $\Delta$. It is not clear to us what one might hope to expect.
\item It is an interesting question whether the assumption of having sufficiently many (depending on $\Delta$ and $k$) vertices of `intermediate' degree, or some variant of it, is necessary or could be weakened.
\item A related question is whether there are any good heuristics to produce $H$. Following the proof strategy that will be used to establish Proposition 2 and applying to the example in Figure 2 leads to the outcome illustrated in Figure 3. However, in that special case, a much shorter solution (shown in Figure 4) exists.
\end{enumerate}

\begin{figure}[ht!]
	\centering
	\begin{subfigure}{\textwidth}
		\centering
		\begin{tikzpicture}[scale=0.75]
			\draw[blue, dashed, very thick] (-6, -3) -- (6, -3);
			\draw[blue, dashed, very thick] (-6, 3) -- (1, 3);
			\foreach \a in {-6,-4,-2,-1,1,2,4,6}{
				\draw[very thick] (\a, 1) -- (\a, -1);
			}
			\draw[very thick] (-5, -1) -- (-6, 1) -- (-3.5, 1);
			\draw[very thick] (-5, 1) -- (-6, -1) -- (-3.5, -1);
			\draw[very thick] (5, -1) -- (6, 1) -- (3.5, 1);
			\draw[very thick] (5, 1) -- (6, -1) -- (3.5, -1);
			\draw[very thick] (-2.5, 1) -- (2.5, 1);
			\draw[very thick] (-2.5, -1) -- (2.5, -1);
			\draw[red, dashed, very thick] (0, -2) -- (0, -1);
			\draw[red, dashed, very thick] (0, 2) -- (0, 1);
			\draw[blue, dashed, very thick] (0, 1) -- (0, -1);
			\draw[very thick] (-6, -3) -- (-6, -2) -- (-3.5, -2);
			\draw[very thick] (6, -3) -- (6, -2) -- (3.5, -2);
			\draw[very thick] (-2.5, -2) -- (2.5, -2);
			\draw[very thick] (-6, 3) -- (-6, 2) -- (-3.5, 2);
			\draw[very thick] (6, 3) -- (6, 2) -- (3.5, 2);
			\draw[very thick] (-2.5, 2) -- (2.5, 2);
			\draw[very thick] (1, 3) -- (2.5, 3);
			\draw[red, dashed, very thick] (1, 2) -- (1, 3);
			\draw[very thick] (3.5, 3) -- (6, 3);
			\foreach \a in {-6,-5,-4,-2,-1,0,1,2,4,5,6}{
				\filldraw (\a, 2) circle (0.1);
				\filldraw (\a, 1) circle (0.1);
				\filldraw (\a, -1) circle (0.1);
				\filldraw (\a, -2) circle (0.1);
			}
			\foreach \a in {1,2,4,5,6}{
				\filldraw (\a, 3) circle (0.1);
			}
			\filldraw[green] (-6, -3) circle (0.1);
			\filldraw[green] (-6, 3) circle (0.1);
			\filldraw[green] (6, -3) circle (0.1);
			\filldraw[green] (1, 3) circle (0.1);
			\filldraw[green] (2, 2) circle (0.1);
			\node at (-3, 1) {\(\cdots\)};
			\node at (3, 1) {\(\cdots\)};
			\node at (-3, -1) {\(\cdots\)};
			\node at (3, -1) {\(\cdots\)};
			\node at (-3, -2) {\(\cdots\)};
			\node at (-3, 2) {\(\cdots\)};
			\node at (3, -2) {\(\cdots\)};
			\node at (3, 2) {\(\cdots\)};
			\node at (3, 3) {\(\cdots\)};
			\node at (-3, 0) {\(\cdots\)};
			\node at (3, 0) {\(\cdots\)};
		\end{tikzpicture}
	\end{subfigure}
	\caption{Another example of a graph \(H\), the union of regular graphs, such that \(d(G, H) = 6\), where $G$ is the example from Figure 2 and Figure 3.}
\end{figure}
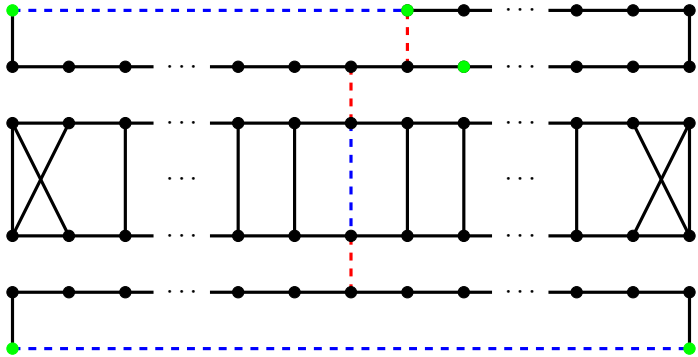

\subsection{Related results.}
We are not aware of any directly related results. One can show that the metric \(d\) on graphs used above is equivalent to graph edit distance, which has been studied in applications such as pattern recognition, unsupervised neural networks, and similarity of seriated graphs, we refer to the survey \cite{survey}. Bunke describes a notion of graph edit distance whose computation is equivalent to the maximum common subgraph problem \cite{subgraph}. It has been shown that graph edit distance computation itself is NP-hard \cite{stars} though some work has been done toward approximations and exact computations \cite{efficient, fast, approximate, speeding}. 
One way of describing our problem is to view graphs as almost locally regular but for a few misplaced edges. In this sense, a locally regular graph \(H\) close to \(G\) can be viewed as an edited decomposition of \(G\) into regular components. Graph decompositions typically refer to a edge or vertex decompositions, possibly into regular or irregular structures \cite{ahadi2018algorithmic, akbari, bensmail2016decomposing, bryant, merker, przybylo2015decomposing}. A key difference between these works and ours is that we are allowed to add or delete a small number of edges.


\section{Proofs of Proposition 1 and 2}
\subsection{Proof of Proposition 1}
Proposition 1 has a very simple direct proof relying on the fact that modifying edges only affects the degrees of relatively few vertices.
\begin{proof}
	Every removed edge decreases the degree of exactly two vertices. Only these two vertices could become new separation vertices, as none of their neighbors could gain a higher degree neighbor. Every added edge increases the degree of exactly two vertices. At worst, all \(2d\) neighbors of these two vertices become new separation vertices. The vertices of the added edge themselves do not become separation vertices if they were not already. Therefore, since we modify \(k\) edges, \(G'\) has no more than \(2dk\) separation vertices.
\end{proof}

\subsection{Proposition 2: Preparatory Results.}
We address Proposition \ref{prop:spirit} by considering maximal degree 3 graphs in three steps.
\begin{itemize}
	\item Graphs with at most two degree 2 vertices and no degree 1 vertices.
	\item Graphs with arbitrarily many degree 2 vertices and no degree 1 vertices.
	\item Graphs with arbitrarily many degree 2 and degree 1 vertices.
\end{itemize}

\begin{lemma}
	Suppose we have an almost regular graph \(G\) where at most two vertices have degree 2, and all other vertices have degree 3. Then there exists a union of regular graphs \(H\) such that 
    \(d(G, H) \le 4\).
	\label{lem:32}
\end{lemma}

\begin{proof}
	If we have no degree 2 vertices, there is nothing to do.
	If we have a single degree 2 vertex, simply delete it, thereby reducing to the case with two degree 2 vertices. 
        This deletion costs 1 operation.
	Now suppose \(G\) has two vertices \(v_1, v_2\) of degree 2. If they are not connected, add an edge between them, and we are done. Otherwise, the idea will be to delete a sufficiently far away edge \((u_1, u_2)\) and add edges \((u_1, v_1)\) and \((u_2, v_2)\). Suppose we have \(2m\) degree 3 vertices. Compute that \(G\) has \((2m \cdot 3 + 2 \cdot 2) / 2 = 3m + 2\) edges, where \(2m\) is the number of degree 3 vertices. We may ignore \(m = 1\) since the only such graph (sometimes called the diamond graph) falls under the previous case. Let \(u_1, u_2\) be the respective higher degree neighbors of \(v_1, v_2\). If we enumerate the edges incident to these labeled vertices, we have \((v_1, v_2), (u_1, v_1), (u_2, v_2)\), and at most four other edges incident to \(u_1\) and \(u_2\). Since we need only consider \(m \ge 2\), \(G\) has at least 8 edges. In particular, there exists an edge in \(G\) not incident to any of the previously labeled vertices. Let this edge be \((w_1, w_2)\). Delete it and add edges \((v_1, w_1)\) and \((v_2, w_2)\). Since \(v_1, v_2\) were the only vertices of degree 2, this yields a 3-regular graph, costing 3 operations.
	In total, we have spent at most 
    4 operations.
\end{proof}

\begin{lemma}
	Suppose we have an almost regular graph \(G\) with at least two vertices of degree 3, at least one vertex of degree 2, and \(k\) separation vertices. Then there exists a union of regular graphs \(H\) such that 
    \(d(G, H) \le 3k + 4\).
	\label{lem:32s}
\end{lemma}

\begin{proof}
	Let \(U\) be the set of separation vertices that have a degree 2 neighbor which is not a separation vertex. We claim that members of \(U\) come in pairs in the following way. Take \(u \in U\), with \(v\) its degree 2 neighbor. Consider the simple path starting from \(u\), going to \(v\), and terminating when we reach another separation vertex \(u'\). This termination condition must occur since \(G\) is finite, we have no degree 1 vertices, and we must encounter a separation vertex before a degree 3 vertex.
    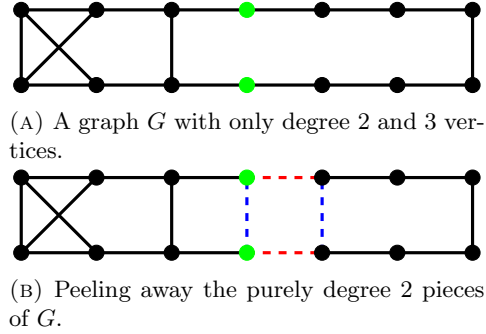
\begin{figure}[!h]
        \centering
        \begin{subfigure}{0.5\textwidth}
            \centering
            \begin{tikzpicture}
                \foreach \a in {0, 1, 2, 4, 5, 6}{
                    \filldraw (\a, 0) circle (0.1);
                    \filldraw (\a, 1) circle (0.1);
                }
                \draw[very thick] (0, 0) -- (1, 0) -- (2, 0) -- (3, 0) -- (4, 0) -- (5, 0) -- (6, 0) -- (6, 1) -- (5, 1) -- (4, 1) -- (3, 1) -- (2, 1) -- (1, 1) -- (0, 1) -- (0, 0);
                \draw[very thick] (0, 0) -- (1, 1);
                \draw[very thick] (1, 0) -- (0, 1);
                \draw[very thick] (2, 0) -- (2, 1);
                \filldraw[green] (3, 0) circle (0.1);
                \filldraw[green] (3, 1) circle (0.1);
            \end{tikzpicture}
            \caption{A graph \(G\) with only degree 2 and 3 vertices.}
        \end{subfigure}
        \begin{subfigure}{0.5\textwidth}
            \centering
            \begin{tikzpicture}
                \draw[very thick, dashed, blue] (3, 0) -- (3, 1);
                \draw[very thick, dashed, blue] (4, 0) -- (4, 1);
                \draw[very thick, dashed, red] (3, 0) -- (4, 0);
                \draw[very thick, dashed, red] (3, 1) -- (4, 1);
                \foreach \a in {0, 1, 2, 4, 5, 6}{
                    \filldraw (\a, 0) circle (0.1);
                    \filldraw (\a, 1) circle (0.1);
                }
                \draw[very thick] (0, 0) -- (1, 1);
                \draw[very thick] (1, 0) -- (0, 1);
                \draw[very thick] (2, 0) -- (2, 1);
                \draw[very thick] (4, 0) -- (5, 0) -- (6, 0) -- (6, 1) -- (5, 1) -- (4, 1);
                \draw[very thick] (3, 0) -- (2, 0) -- (1, 0) -- (0, 0) -- (0, 1) -- (1, 1) -- (2, 1) -- (3, 1);
                \filldraw[green] (3, 0) circle (0.1);
                \filldraw[green] (3, 1) circle (0.1);
            \end{tikzpicture}
            \caption{Peeling away the purely degree 2 pieces of \(G\).}
        \end{subfigure}
        \caption{Addressing degree 2 separation vertices with a non-separation degree 2 neighbor in the manner described in the proof of Lemma \ref{lem:32s}.}
    \end{figure}
	Let \(v'\) be the degree 2 neighbor of \(u'\) (noting \(v = v'\) is possible). Delete edges \((u, v)\) and \((u', v')\), and add edge \((u, u')\). This converts \(u, u'\) into (neighboring) separation vertices, and leaves \(u', v'\) as either an isolated vertex or a path component. Connect \(v', v\) with an edge if the path component contains more than 2 vertices. This entire procedure costs at most \(4k / 2 = 2k\) operations. Notice that the degree 3 neighbors of \(u, u'\) are unaffected; in particular, we have introduced no new separation vertices.
	Now, the only degree 2 vertices remaining are separation vertices. Observe that we cannot have three separation vertices all adjacent to each other. This is because such a structure would need to be an isolated \(K_3\) which violates the fact that they need to be separation vertices. Therefore, so long as we have at least 3 separation vertices remaining, there always exists a pair of vertices between which we can add an edge. This allows us to use at most \(k / 2\) operations to leave us with either 1 or 2 separation vertices remaining. We may then apply Lemma \ref{lem:32} on the component containing the remaining separation vertices (this is important so as not to introduce new separation vertices), leaving us with a component-wise regular graph.
	In total, we use no more than 
    \(2k + k / 2 + 4 \le 3k + 4\) operations.
\end{proof}

    We note that the ideas here make use of a concept sometimes called an edge swap \cite{swap}, this concept is illustrated in Figure \ref{fig:swap}.

\begin{figure}[!h]
    \centering
    \begin{subfigure}{0.4\textwidth}
        \centering
        \begin{tikzpicture}
            \filldraw (0, 0) circle (0.1);
            \filldraw (1, 0) circle (0.1);
            \filldraw (2, 0) circle (0.1);
            \filldraw (3, 0) circle (0.1);
            \filldraw (0, 1) circle (0.1);
            \filldraw (1, 1) circle (0.1);
            \filldraw (2, 1) circle (0.1);
            \filldraw (3, 1) circle (0.1);
            \draw[very thick] (0, 0) -- (1, 0) -- (2, 0) -- (3, 0) -- (3, 1) -- (2, 1) -- (1, 1) -- (0, 1) -- (0, 0);
        \end{tikzpicture}
        \caption{A 2-regular graph.}
    \end{subfigure}
    \begin{subfigure}{0.4\textwidth}
        \centering
        \begin{tikzpicture}
            \draw[very thick, dashed, blue] (1, 0) -- (1, 1);
            \draw[very thick, dashed, blue] (2, 0) -- (2, 1);
            \draw[very thick, dashed, red] (1, 0) -- (2, 0);
            \draw[very thick, dashed, red] (1, 1) -- (2, 1);
            \filldraw (0, 0) circle (0.1);
            \filldraw (1, 0) circle (0.1);
            \filldraw (2, 0) circle (0.1);
            \filldraw (3, 0) circle (0.1);
            \filldraw (0, 1) circle (0.1);
            \filldraw (1, 1) circle (0.1);
            \filldraw (2, 1) circle (0.1);
            \filldraw (3, 1) circle (0.1);
            \draw[very thick] (1, 0) -- (0, 0) -- (0, 1) -- (1, 1);
            \draw[very thick] (2, 0) -- (3, 0) -- (3, 1) -- (2, 1);
        \end{tikzpicture}
        \caption{The result of an edge swap.}
    \end{subfigure}
    \caption{Notice that in a regular graph, an edge swap does not change the degree of any vertex.}
    \label{fig:swap}
\end{figure}
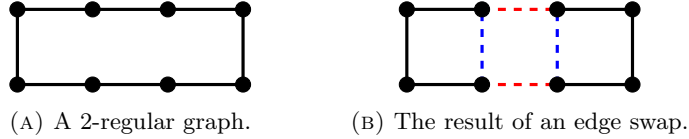

\subsection{Proof of Proposition \ref{prop:spirit}}

\begin{proof}
	\(G\) has at most \(k_1 \le k\) separation vertices of degree 1. Arbitrarily pair up \(2\lfloor k_1 / 2 \rfloor\) of them and add an edge within each pair. This is possible since degree 1 separation vertices cannot be neighbors. If \(k_1\) is even, we are done. Otherwise, we have one more degree 1 separation vertex \(u\), and there exists another vertex \(v\) of odd degree (strictly greater than 1). Delete an edge \((v, w)\) (where \(w \ne u\) if \((v, u)\) exists), and add the edge \((u, w)\). This procedure of addressing degree 1 separation vertices requires at most \(k / 2 + 2\) operations.
	The neighbors of the first \(2\lfloor k / 2 \rfloor\) vertices we considered all had degree at least 2, so if any were not separation vertices, they certainly do not become separation vertices. Only vertex \(v\) might become a new separation vertex since its degree decrements by one. Consequently, our new graph has at most \(k + 1\) separation vertices. Consider the union of components that do not contain any degree 1 vertices. By Lemma \ref{lem:32s}, we need 
    \(3(k + 1) + 4\)
    operations to reach a locally regular graph. 
    We have used no more than \(4k + 9\) operations.
\end{proof}

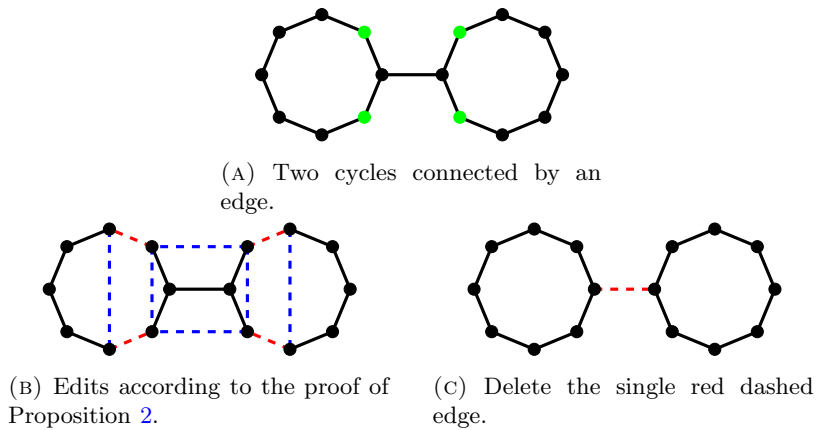
\begin{figure}[!h]
	\centering
	\begin{subfigure}{0.4\textwidth}
		\centering
		\begin{tikzpicture}[scale=0.8]
			\foreach \a in {2,3,4,5,6,8}{
                    \filldraw (360 / 8 * \a:1) circle (0.1) node (\a) {};
			}
                \foreach \a in {1,2,4,6,7,8}{
                    \filldraw (3, 0) + (360 / 8 * \a:1) circle (0.1);
                }
			\foreach \a in {1,...,8}{
				\draw[very thick] (0, 0) + (360 / 8 * \a:1) -- (360 / 8 * \a + 45:1);
				\draw[very thick] (0:3) + (360 / 8 * \a:1) -- +(360 / 8 * \a + 45:1);
			}
			\draw[very thick] (1, 0) -- (2, 0);
                \filldraw[green] (360 / 8 * 1:1) circle (0.1);
                \filldraw[green] (360 / 8 * 7:1) circle (0.1);
                \filldraw[green] (3, 0) + (360 / 8 * 5:1) circle (0.1);
                \filldraw[green] (3, 0) + (360 / 8 * 3:1) circle (0.1);
		\end{tikzpicture}
		\caption{Two cycles connected by an edge.}
	\end{subfigure}

	\begin{subfigure}{0.4\textwidth}
		\centering
		\begin{tikzpicture}[scale=0.8]
			\draw[blue, dashed, very thick] (0, 0) + (360 / 8 * 1:1) -- (360 / 8 * 6 + 45:1);
			\draw[blue, dashed, very thick] (0, 0) + (360 / 8 * 2:1) -- (360 / 8 * 5 + 45:1);
			\draw[blue, dashed, very thick] (0:3) + (-360 / 8 * 1 + 180 :1) -- +(-360 / 8 * 6 + 135:1);
			\draw[blue, dashed, very thick] (0:3) + (-360 / 8 * 2 + 180 :1) -- +(-360 / 8 * 5 + 135:1);
			\draw[blue, dashed, very thick] (3, 0) + (360 / 8 * 5:1) -- (360 / 8 * 6 + 45:1);
			\draw[blue, dashed, very thick] (3, 0) + (360 / 8 * 3:1) -- (360 / 8 * 8 + 45:1);
			\draw[red, dashed, very thick] (0, 0) + (360 / 8 * 1:1) -- (360 / 8 * 1 + 45:1);
			\draw[red, dashed, very thick] (0, 0) + (360 / 8 * 6:1) -- (360 / 8 * 6 + 45:1);
			\draw[red, dashed, very thick] (0:3) + (-360 / 8 * 1 + 180 :1) -- +(-360 / 8 * 1 + 135:1);
			\draw[red, dashed, very thick] (0:3) + (-360 / 8 * 6 + 180 :1) -- +(-360 / 8 * 6 + 135:1);
			\foreach \a in {1,...,8}{
				\filldraw (360 / 8 * \a:1) circle (0.1) node (\a) {};
				\filldraw (3, 0) + (360 / 8 * \a:1) circle (0.1);
			}
			\foreach \a in {2,3,4,5,7,8}{
				\draw[very thick] (0, 0) + (360 / 8 * \a:1) -- (360 / 8 * \a + 45:1);
				\draw[very thick] (0:3) + (-360 / 8 * \a + 180 :1) -- +(-360 / 8 * \a + 135:1);
			}
			\draw[very thick] (1, 0) -- (2, 0);
		\end{tikzpicture}
		\caption{Edits according to the proof of Proposition \ref{lem:32s}.}
	\end{subfigure}
	\(\quad\)
	\begin{subfigure}{0.4\textwidth}
		\centering
		\begin{tikzpicture}[scale=0.8]
			\draw[red, dashed, very thick] (1, 0) -- (2, 0);
			\foreach \a in {1,...,8}{
				\filldraw (360 / 8 * \a:1) circle (0.1) node (\a) {};
				\filldraw (3, 0) + (360 / 8 * \a:1) circle (0.1);
			}
			\foreach \a in {1,...,8}{
				\draw[very thick] (0, 0) + (360 / 8 * \a:1) -- (360 / 8 * \a + 45:1);
				\draw[very thick] (0:3) + (360 / 8 * \a:1) -- +(360 / 8 * \a + 45:1);
			}
		\end{tikzpicture}
		\caption{Delete the single red dashed edge.}
	\end{subfigure}
	\caption{Two locally regular graphs close to a given graph.}
\end{figure} 
In some cases, while Proposition \ref{prop:spirit} gives something not far from optimal, it seems to miss the truth. Consider (see Fig. 5) the graph on \(2n\) vertices consisting of two \(n\)-cycles connected by a single edge, for large \(n\). Our procedure will give 10 operations, resulting in two \((n - 3)\)-cycles, and a 6-vertex 3-regular component. 
This completely misses the obvious (and much more elegant) choice of simply deleting the single edge connecting the two cycles to yield two disjoint \(n\)-cycles. This makes some sense, as this graph is almost regular in the sense that it is almost 2-regular, while Proposition \ref{prop:spirit} is designed to handle almost 3-regular graphs.
One cannot easily amend the algorithm to this example, as the edge we would like to delete is not incident to a separation vertex. What this seems to suggest is that our notion of separation is incomplete. It does not seem to be able to nicely detect whether certain parts of a graph are `essentially' degree 2 or `essentially' degree 3. It might be necessary to consider vertices with neighbors of lower degree as well.\\

One interesting feature of this proof is that it very nearly does not depend on the non-separation vertices being degree 3. In fact, the only place we use that is in the application of Lemma \ref{lem:32}, when we show that there exists a sufficiently distant edge we can delete. Morally speaking, having vertices of arbitrary degree should not affect the existence of a sufficiently distant edge for us to delete. We use this idea to generalize to graphs of arbitrary maximal degree.

\section{Proof of Main Result}

The proof of Proposition \ref{prop:spirit} uses the idea of cutting away all portions of the graph that are locally 2-regular to split the graph into a locally 3-regular and locally 2-regular portion. This makes analysis more straightforward, and we generalize this approach for our main result. 

\subsection{Preparatory Statements.}

\begin{lemma}
	Suppose we have a graph \(G = (V, E)\) whose vertices have max degree \(\Delta\). 
    If \(G\) has \(k\) separation vertices, there exists a union of two disjoint graphs \(H = G_- \cup G_\Delta\) such that the following are true.
    \begin{enumerate}
        \item All vertices of \(G_\Delta\) with degree less than \(\Delta\) are separation vertices.
        \item \(G_\Delta\) has at most \(k\) separation vertices.
        \item All vertices of \(G_-\) have degree strictly less than \(\Delta\).
        \item \(G_-\) has at most \(k(\Delta - 1)\) separation vertices.
        \item \(d(G, H) \le k(\Delta - 2)\).
    \end{enumerate}
	\label{lem:cut}
\end{lemma}

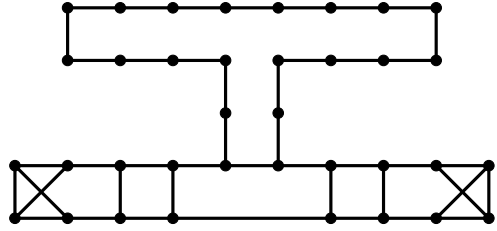
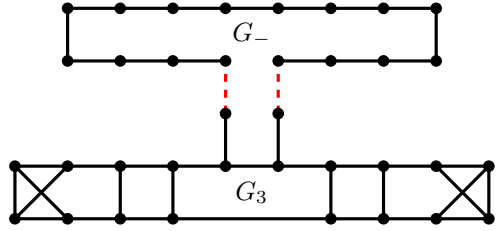
\begin{figure}[ht!]
	\centering
	\begin{subfigure}{\textwidth}
		\centering
		\begin{tikzpicture}[scale=0.7]
			\draw[very thick] (-4, 0) -- (5, 0);
			\draw[very thick] (-4, 1) -- (5, 1);
			\foreach \a in {-4,-2,-1,2,3,5}{
				\draw[very thick] (\a, 0) -- (\a, 1);
			}
			\draw[very thick] (-4, 0) -- (-3, 1);
			\draw[very thick] (-4, 1) -- (-3, 0);
			\draw[very thick] (4, 0) -- (5, 1);
			\draw[very thick] (4, 1) -- (5, 0);
			\draw[very thick] (0, 1) -- (0, 3);
			\draw[very thick] (1, 1) -- (1, 3);
			\draw[very thick] (0, 3) -- (-3, 3) -- (-3, 4) -- (4, 4) -- (4, 3) -- (1, 3);
			\foreach \a in {-4,-3,-2,-1,2,3,4,5}{
				\filldraw (\a, 0) circle (0.1);
				\filldraw (\a, 1) circle (0.1);
			}
			\foreach \a in {-3,...,4}{
				\filldraw (\a, 3) circle (0.1);
				\filldraw (\a, 4) circle (0.1);
			}
			\filldraw (0, 1) circle (0.1);
			\filldraw (1, 1) circle (0.1);
			\filldraw (0, 2) circle (0.1);
			\filldraw (1, 2) circle (0.1);
		\end{tikzpicture}
		\caption{A graph \(G\) with \(\Delta = 3\).}
	\end{subfigure}

	\begin{subfigure}{\textwidth}
		\centering
		\begin{tikzpicture}[scale=0.7]
			\draw[very thick] (-4, 0) -- (5, 0);
			\draw[very thick] (-4, 1) -- (5, 1);
			\foreach \a in {-4,-2,-1,2,3,5}{
				\draw[very thick] (\a, 0) -- (\a, 1);
			}
			\draw[very thick] (-4, 0) -- (-3, 1);
			\draw[very thick] (-4, 1) -- (-3, 0);
			\draw[very thick] (4, 0) -- (5, 1);
			\draw[very thick] (4, 1) -- (5, 0);
			\draw[very thick] (0, 1) -- (0, 2);
			\draw[very thick] (1, 1) -- (1, 2);
			\draw[red, dashed, very thick] (0, 2) -- (0, 3);
			\draw[red, dashed, very thick] (1, 2) -- (1, 3);
			\draw[very thick] (0, 3) -- (-3, 3) -- (-3, 4) -- (4, 4) -- (4, 3) -- (1, 3);
			\foreach \a in {-4,-3,-2,-1,2,3,4,5}{
				\filldraw (\a, 0) circle (0.1);
				\filldraw (\a, 1) circle (0.1);
			}
			\foreach \a in {-3,...,4}{
				\filldraw (\a, 3) circle (0.1);
				\filldraw (\a, 4) circle (0.1);
			}
			\filldraw (0, 1) circle (0.1);
			\filldraw (1, 1) circle (0.1);
			\filldraw (0, 2) circle (0.1);
			\filldraw (1, 2) circle (0.1);
			\node at (0.5, 0.5) {\(G_3\)};
			\node at (0.5, 3.5) {\(G_-\)};
		\end{tikzpicture}
		\caption{A graph \(G'\) obtained from \(G\) via Lemma \ref{lem:cut}.}
	\end{subfigure}
	\caption{An illustration of Lemma \ref{lem:cut} in the \(\Delta = 3\) case.}
\end{figure}

\begin{proof} 
	Let \(V\) denote the set of all separation vertices adjacent to a vertex with degree \(\Delta\). Delete all edges of the form \((u, v)\) where \(u \not \in V\), \(\deg u < \Delta\), and \(v \in V\).
	Consider the subgraph \(G_\Delta\) of the resulting graph containing all top degree vertices and their immediate neighbors. Any vertex of degree less than \(\Delta\) in \(G_\Delta\) must be a separation vertex adjacent to a top degree vertex. Let \(V_1\) denote the set of such vertices in \(G_\Delta\). For any edge \((v_1, v)\) with \(v_1 \in V_1\), we then must have that \(\deg v = \Delta\) or \(v \in V_1\). All separation vertices in \(G_\Delta\) were separation vertices in our original graph \(G\), so \(G_\Delta\) has at most \(k\) separation vertices. 
	Let subgraph \(G_-\) contain all other vertices. By construction, there are no edges between \(G_\Delta\) and \(G_-\). Our edge deletions introduce at most \(k(\Delta - 2)\) new separation vertices, so \(G_-\) has no more than \(k(\Delta - 1)\) separation vertices. By construction, \(G_-\) has no vertex with degree \(\Delta\).
	All deletions occur only on edges incident to an original separation vertex (though not to a top degree vertex), so we have used at most \(k(\Delta - 2)\) operations.
\end{proof}

We remark that \(G_\Delta\) satisfies the assumptions of the following proposition.

\begin{proposition}
	Suppose we have a graph \(G = (V, E)\) whose vertices have max degree \(\Delta\). Further suppose that any vertex of degree strictly less than \(\Delta\) is a separation vertex and that \(G\) has \(k\) separation vertices. Then if \(G\) has at least \(2\Delta + \Delta^2 + \Delta^3\) vertices, we require at most \(3k\Delta\) operations to obtain an almost regular graph whose vertices all have degree \(\Delta\) or \(\Delta - 1\). 
	\label{prop:almostregn}
\end{proposition}

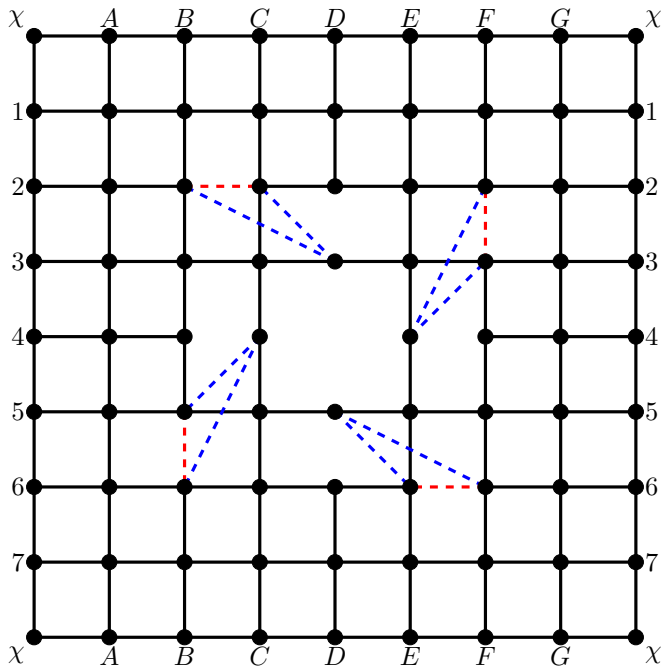
\begin{figure}[ht!]
    \centering
    \begin{tikzpicture}
        \foreach \a in {-4,-3,-1,1,3,4}{
            \draw[very thick] (\a, -4) -- (\a, 4);
            \draw[very thick] (-4, \a) -- (4, \a);
        }
        \draw[very thick] (-4, 0) -- (-2, 0);
        \draw[very thick] (4, 0) -- (2, 0);
        \draw[very thick] (0, 4) -- (0, 2);
        \draw[very thick] (0, -4) -- (0, -2);
        \draw[blue, dashed, very thick] (-1, 0) -- (-2, -1);
        \draw[blue, dashed, very thick] (-1, 0) -- (-2, -2);
        \draw[blue, dashed, very thick] (0, 1) -- (-2, 2);
        \draw[blue, dashed, very thick] (0, 1) -- (-1, 2);
        \draw[blue, dashed, very thick] (1, 0) -- (2, 1);
        \draw[blue, dashed, very thick] (1, 0) -- (2, 2);
        \draw[blue, dashed, very thick] (0, -1) -- (2, -2);
        \draw[blue, dashed, very thick] (0, -1) -- (1, -2);
        \draw[red, dashed, very thick] (-2, -1) -- (-2, -2);
        \draw[red, dashed, very thick] (-2, 2) -- (-1, 2);
        \draw[red, dashed, very thick] (2, 1) -- (2, 2);
        \draw[red, dashed, very thick] (2, -2) -- (1, -2);
        \draw[very thick] (-4, 2) -- (-2, 2);
        \draw[very thick] (-1, 2) -- (4, 2);
        \draw[very thick] (2, 4) -- (2, 2);
        \draw[very thick] (2, 1) -- (2, -4);
        \draw[very thick] (4, -2) -- (2, -2);
        \draw[very thick] (1, -2) -- (-4, -2);
        \draw[very thick] (-2, -4) -- (-2, -2);
        \draw[very thick] (-2, -1) -- (-2, 4);
        \foreach \a in {-4,...,4}{
            \foreach \b in {-4,...,4}{
                \filldraw (\a, \b) circle (0.1);
            }
        }
        \filldraw[white] (0, 0) circle (0.2);
        \node[anchor=south west] at (4, 4) {\(\chi\)};
        \node[anchor=south east] at (-4, 4) {\(\chi\)};
        \node[anchor=north west] at (4, -4) {\(\chi\)};
        \node[anchor=north east] at (-4, -4) {\(\chi\)};
        \node[anchor=east] at (-4, 3) {\(1\)};
        \node[anchor=east] at (-4, 2) {\(2\)};
        \node[anchor=east] at (-4, 1) {\(3\)};
        \node[anchor=east] at (-4, 0) {\(4\)};
        \node[anchor=east] at (-4, -1) {\(5\)};
        \node[anchor=east] at (-4, -2) {\(6\)};
        \node[anchor=east] at (-4, -3) {\(7\)};
        \node[anchor=west] at (4, 3) {\(1\)};
        \node[anchor=west] at (4, 2) {\(2\)};
        \node[anchor=west] at (4, 1) {\(3\)};
        \node[anchor=west] at (4, 0) {\(4\)};
        \node[anchor=west] at (4, -1) {\(5\)};
        \node[anchor=west] at (4, -2) {\(6\)};
        \node[anchor=west] at (4, -3) {\(7\)};
        \node[anchor=south] at (-3, 4) {\(A\)};
        \node[anchor=south] at (-2, 4) {\(B\)};
        \node[anchor=south] at (-1, 4) {\(C\)};
        \node[anchor=south] at (0, 4) {\(D\)};
        \node[anchor=south] at (1, 4) {\(E\)};
        \node[anchor=south] at (2, 4) {\(F\)};
        \node[anchor=south] at (3, 4) {\(G\)};
        \node[anchor=north] at (-3, -4) {\(A\)};
        \node[anchor=north] at (-2, -4) {\(B\)};
        \node[anchor=north] at (-1, -4) {\(C\)};
        \node[anchor=north] at (0, -4) {\(D\)};
        \node[anchor=north] at (1, -4) {\(E\)};
        \node[anchor=north] at (2, -4) {\(F\)};
        \node[anchor=north] at (3, -4) {\(G\)};
    \end{tikzpicture}
    \caption{An illustration of Proposition \ref{prop:almostregn} for a graph with \(\Delta = 4\). Note that the vertices on the edges are identified according to their labels. Technically this graph does not satisfy the assumptions of Proposition \ref{prop:almostregn}, but the assumptions are not a necessary condition.}
\end{figure}

\begin{proof}
	Observe that we cannot have \(\Delta\) separation vertices all mutually adjacent to each other. This means we can always add edges between separation vertices until we have at most \(\Delta - 1\) separation vertices all mutually adjacent. There are two things to note about this statement.
	\begin{itemize}
		\item If we start with \(k > \Delta - 1\) separation vertices, much like in the proof of Lemma \ref{lem:32s}, this procedure of adding edges between separation vertices necessarily turns some separation vertices into degree \(\Delta\) vertices.
		\item The remaining separation vertices may have any degree \(d < \Delta\).
	\end{itemize}
	The added edges cost at most \(k\Delta\) operations.

	We proceed as follows. Let \(V_0\) denote all remaining separation vertices, with \(k_0 = |V_0|\). Let \(V_1\) denote the neighbors of \(V_0\), excluding elements of \(V_0\) itself. Let \(V_2\) denote the neighbors of \(V_1\), excluding elements of both \(V_0\) and \(V_1\). Notice that
	\[ |V_0| = k_0 \le \Delta - 1, \quad |V_1| \le k_0(\Delta - 1) \le (\Delta - 1)^2, \quad |V_2| \le k_0(\Delta - 1)^2 \le (\Delta - 1)^3. \]
	Edges within incident to these sets could be too close to the separation vertices. Any edge of \(G\) with at least one vertex not in \(V_0, V_1, V_2\), however, is guaranteed to have both its vertices not adjacent to any separation vertex. We would like for there to exist \(k_0\Delta / 2\) edges in \(G\) sufficiently far from any of the separation vertices. If \(G\) has \(x\) additional vertices, there are at least \(\Delta x / 2\) additional edges. Setting \(x = k_0\) is sufficient. The following applies so long as \(G\) has at least \(2\Delta + \Delta^2 + \Delta^3 \ge k_0 + k_0(\Delta - 1) + k_0(\Delta - 1)^2 + k_0\) vertices. 

	For each separation vertex \(v\) with degree \(d\) strictly less than \(\Delta - 1\), let \(l = \lfloor (\Delta - 1 - d) / 2 \rfloor\). Delete \(l\) sufficiently distant edges \(u_i = (u^{(1)}_i, u^{(2)}_i)\) and add edges \((v, u^{(1)}_i)\) and \((v, u^{(2)}_i)\), which costs \(3l\) operations. In total, this costs no more than \(k_0 \cdot 3\Delta / 2\) operations. Notice that this is possible since \(l \le \Delta / 2\), and we ensured that \(G\) has at least \(k_0\Delta / 2\) sufficiently far away edges. Following this, every separation vertex will have become a degree \(\Delta\) vertex or a degree \(\Delta - 1\) separation vertex. None of the originally degree \(\Delta\) vertices have had their degrees modified.

	Enumerating all our operations yields \(k\Delta + 3k\Delta / 2 \le 3k\Delta\).
\end{proof}

\begin{proposition}
	Suppose we have an almost regular graph \(G = (V, E)\) whose vertices have degree \(\Delta\) or \(\Delta - 1\) and that \(G\) has \(k\) vertices of degree \(\Delta - 1\), all of which are separation vertices.
 Then if \(G\) has at least \(\Delta + \Delta^2 + \Delta^3 + 2\) vertices, \(G\) is at most \(k + 4\Delta\) operations away from a component-wise regular graph. 
	\label{prop:regn}
\end{proposition}

\begin{proof}
	The opening of this proof will follow similarly to the beginning of the proof of Proposition \ref{prop:almostregn}. Observe that we cannot have \(\Delta\) vertices of degree \(\Delta - 1\) all adjacent to each other. This is because we would have an isolated \(K_\Delta\) in \(G\), which would violate these degree \(\Delta - 1\) vertices being separation vertices. Therefore, if we have at least \(\Delta\) separation vertices, there always exists a pair of separation vertices between which we can add an edge. This allows us to use at most \(k / 2\) operations to at most \(\Delta - 1\) separation vertices of degree \(\Delta - 1\), all mutually adjacent.

	Let \(k_0 \le \Delta - 1\) denote the number of remaining separation vertices. Once again, the idea is to delete sufficiently far away edges and connect the separation vertices to the vertices of the just deleted edges. We need to check that \(G\) has sufficiently far away edges. Let \(V_0\) denote all separation vertices. Let \(V_1\) denote the neighbors of \(V_0\), excluding elements of \(V_0\) itself. Let \(V_2\) denote the neighbors of \(V_1\) (excluding elements of both \(V_0\) and \(V_1\)). Notice that
	\[ |V_0| = k_0 \le \Delta - 1, \quad |V_1| \le k_0(\Delta - 1) \le (\Delta - 1)^2, \quad |V_2| \le k_0(\Delta - 1)^2 \le (\Delta - 1)^3. \]
	Edges incident to these sets could be too close to the separation vertices. Any edge of \(G\) with at least one vertex not in \(V_0, V_1, V_2\), however, is guaranteed to have both its vertices not adjacent to any separation vertex. We would like \(G\) to have at least \(k_0 / 2\) sufficiently far edges. If \(G\) has \(x\) additional vertices, there are at least \(\Delta x / 2\) additional edges. Setting \(x = 2\) is sufficient.

	Therefore, if \(G\) has at least \(\Delta + \Delta^2 + \Delta^3 + 2 \ge k_0 + k_0(\Delta - 1) + k_0(\Delta - 1)^2 + 2\) vertices, we may proceed with the following procedure. Arbitrarily pair up separation vertices. For each pair \(u_1, u_2\), delete a sufficiently far edge \((v_1, v_2)\). Add edges \((u_1, v_1)\) and \((u_2, v_2)\). This costs at most \(3(\Delta - 1) / 2 \ge 3k_0 / 2\) operations. If \(k_0\) is even, we are done. Otherwise, we have one remaining separation vertex (which also implies that \(\Delta\) is odd). Delete all edges incident to this remaining separation vertex (costing \(\Delta\) operations). Then \(G\) is left with \(\Delta - 1\) separation vertices of degree \(\Delta - 1\). This allows us to go back, and repeat this previously described procedure once, costing at most \(3(\Delta - 1) / 2\) operations. 

        In total, we have used no more than \(k + 4\Delta\) operations.
\end{proof}


\subsection{Proof of the Main Result}
\begin{proof}
	We proceed by induction. As a base case, recall from Proposition \ref{prop:spirit} that any graph with max degree 3 is at most 
    \(4k + 9\) operations from a component-wise regular graph.

	Apply Lemma \ref{lem:cut}, costing \(k(\Delta - 2)\) operations. This yields \(G_\Delta\), a graph with max degree \(\Delta\) such that all low degree vertices are separation vertices, and \(G_-\), a graph with max degree strictly less than \(\Delta\). Since \(G_\Delta\) has at most \(k\) separation vertices and has at least \(2\Delta + \Delta^2 + \Delta^3\) vertices (by assumption), by Proposition \ref{prop:almostregn}, it costs \(3k\Delta\) operations to obtain an almost regular graph with all vertices of degree \(\Delta, \Delta - 1\). Note that the only degree \(\Delta - 1\) vertices that arise in this manner must come from one of the \(k\) separation vertices of \(G_\Delta\). Then by Proposition \ref{prop:regn}, we need no more than \(k + 4\Delta\) operations from a \(\Delta\)-regular graph. 

	It remains to address \(G_-\). Lemma \ref{lem:cut} causes \(G_-\) to have at most \(k(\Delta - 1)\) separation vertices. By assumption, \(G_-\) has at least \(2d + d^2 + d^3 + k\Delta! - k(\Delta - 2)\) vertices of each degree \(d < \Delta\), where the subtraction term arises from Lemma \ref{lem:cut} reducing the degrees of at most \(k(\Delta - 2)\) vertices. Let \(K = k(\Delta - 1)\), \(d_- = d - 1\), and \(\Delta_- = \Delta - 1\). \(G_-\) has max degree \(\Delta_-\) and at most \(K\) separation vertices. Then notice that for all \(d < \Delta\),
        \begin{align*}
            2d + d^2 + d^3 + k\Delta! - k(\Delta - 2) &> 2d_- + d_-^2 + d_-^3 + k\left( \Delta! - (\Delta - 2) \right) \\
            &> 2d_- + d_-^2 + d_-^3 + k\left( \Delta! - (\Delta - 1)! \right) \\
            &= 2d_- + d_-^2 + d_-^3 + k(\Delta - 1)! \cdot (\Delta - 1) \\
            &= 2d_- + d_-^2 + d_-^3 + K(\Delta_-)!
        \end{align*}
	which shows that \(G_-\) itself satisfies our hypothesis. We have so far used \(k(\Delta - 1)\) operations to form \(G_\Delta, G_-\) and \(3k\Delta + k + 4\Delta\) operations to address \(G_\Delta\). Summing these yields
	\[ k(\Delta - 1) + 3k\Delta + k + 4\Delta = 4k\Delta + 4\Delta. \]
	By induction, we apply the same argument to \(G_-\). Since we have an explicit relation between the maximal degrees of \(G\) and \(G_-\) and between the number of separation vertices of \(G\) and \(G_-\), we may compute that this entire procedure requires
	\[ 4k + 9 + \sum_{i = 0}^{\Delta - 4} \left( (\Delta - i) \cdot 4k \prod_{j = 1}^i (\Delta - j) + 4(\Delta - i) \right) \]
	\[ \le 4k + 9 + \sum_{i = 0}^{\Delta - 4} 4k \Delta! + 4\Delta \le 4k + 9 + 4k\Delta \cdot \Delta! + 4\Delta^2 \]
	\[ = \left( 4 + 4\Delta \cdot \Delta! \right)k + (9 + 4\Delta^2) \]
	operations. Since \(\Delta \ge 4\), we may bound our operations by the cleaner expression
    \[ 4(\Delta + 1)! \cdot k + 5\Delta^2. \]
\end{proof}

The assumption of Theorem \ref{thm:main} requires that our input graph have enough vertices of all degrees less than or equal to \(\Delta\). This condition arises from guaranteeing that there exist sufficiently many distance edges to delete for Propositions \ref{prop:almostregn} and \ref{prop:regn}, but it is difficult to gauge how necessary this condition is.

\subsection*{Acknowledgements} We acknowledge discussions with Stefan Steinerberger.

\bibliographystyle{plain}
\bibliography{refs}

\end{document}